\documentclass{article}
\usepackage{amsmath}
\usepackage{amsthm}
\usepackage{amssymb}
\usepackage{graphicx}
\usepackage{xcolor}

\newtheorem{theorem}{Theorem}[section]
\newtheorem{proposition}[theorem]{Proposition}
\newtheorem{lemma}[theorem]{Lemma}
\newtheorem{corollary}[theorem]{Corollary}
\newtheorem{definition}[theorem]{Definition}
\newtheorem{example}[theorem]{Example}
\newtheorem{remark}[theorem]{Remark}

\begin{document}

\title{Constant angle null hypersurfaces}

\author{Samuel Chable-Naal, Matias Navarro, Didier A. Solis}

\date{\today}

\maketitle

\abstract{In this work we introduce the notion of constant angle null hypersurface of a Lorentzian manifold  with respect to a given ambient vector field. We analyze the case in which the vector field is closed and conformal, thus finding that such null hypersurfaces have a canonical principal direction. We further provide some classification results for constant angle null surfaces with vanishing null mean curvature.
}

\medskip


\medskip


\section{Introduction}\label{sec1}

The classification of submanifolds subject to certain geometric restrictions has been one of the main driving forces in semi-Riemannian geometry. The characterization of totally geodesic, totally umbilical, minimal or isoparametric submanifolds of specific classes of ambient spaces are active areas of research up to this day (for detailed accounts refer to \cite{CecilRyan,chen00,DajczerTojeiro}). Historically, among the first such objects to be explored are the so-called constant angle submanifolds, its study dating back to the works of J. Bernoulli  on spirals and rhumb lines \cite{Bernoulli}. 

In the Riemannian case, given an oriented immersed hypersurface $M$ of a Riemannian manifold $(\bar{M},\bar{g})$ and a distinguished vector field $V\in\Gamma (T\bar{M})$, $V\neq 0$, the \textit{angle} $\theta$ of $M$ with respect to $V$ is defined by 
\begin{equation}\label{eq:angle}
\cos \theta =  \frac{\bar{g} (V,\mathbf{n})}{\vert V\vert},
\end{equation}
 where $\mathbf{n}$ is a unit normal vector field to $M$. Thus, a constant angle hypersurface is characterized by the constancy of $\theta$. This notion can be straightforwardly generalized to the semi-Riemannian context  by requiring the right-hand size in Eq. \eqref{eq:angle}  to be constant along $M$. Notice that this latter expression carries a distinct geometrical meaning in some specific settings. For instance, in the Lorentzian case, it can be interpreted in terms of a hyperbolic cosine, when $M$ is spacelike and $V$ is timelike \cite{MR719023}.  From the geometrical point of view, the most meaningful properties arise in the cases in which the vector field $V$ carries some relevant geometrical information. Classical examples of such fields include those that arise as infinitesimal generators of  geometrical transformations like isometries (Killing fields) or conformal maps (closed conformal or concircular fields). It is worthwhile mentioning that the latter class include parallel, radial and gradient vector fields \cite{KUHNEL}.

Several classification results for constant angle semi-Riemannian (non-degenerate) hypersurfaces have been established recently in  ambient spaces such as cartesian products \cite{Dchen,Dillen01,Dillen02,FuNistor}, warped products \cite{Dillen00,MR3013428}, spaceforms \cite{FuYang,Manfio,MunteanuNistor02,NRS01} and other geometrically relevant spaces \cite{aguilar,Onnis01,Onnis02}. One remarkable relation between the distinguished vector field $V$ and the geometry of a constant angle hypersurface $M$ is encapsulated in the notion of canonical principal direction: the tangent component $V^{\top}\in\Gamma (TM)$ is a principal direction of $M$ \cite{chen01,GG,MR2966642,MunteanuNistor01}. 

We notice that the above discussion deals exclusively with non-degenerate hypersurfaces. Nevertheless, some of the most studied submanifolds in Lorentzian geometry and mathematical relativity ---such as event and Killing horizons--- are null, that is, degenerate \cite{MR1172768,MR0424186,MR757180}. In spite of its  relevance in theoretical physics, a framework for the study of null submanifolds in a spirit close to the classical Riemannian submanifold theory was missing until the 1980's. The main difficulty to be surmounted in this effort lies in the fact that in the realm of Lorentzian geometry the vector fields orthogonal to a null hypersurface $M\subset \bar{M}$ are also tangent to it. Hence, there is no canonical orthogonal splitting of $T\bar{M}$ into normal and tangent components. In \cite{MR1383318} K. Duggal and A. Bejancu  established the foundations for a null submanifold theory based on the choice of an additional structure, the so called screen distribution $S(TM)$. A considerable amount of research has developed following this approach  (see for instance \cite{MR2598375} and references therein). In particular, the study of null hypersurfaces in generalized Robertson-Walker spacetimes is an active area of research to this day \cite{Gelocor}.

\section{Preliminaries}\label{sec2}

Following \cite{MR1383318,MR2598375}, we denote by $(\bar M,\bar g)$  an $(n+2)$-dimensional Lorentzian manifold with metric $\bar g$ and by $M$  an $(n+1)$-manifold immersed in $\bar M$ with degenerate induced metric $g$. Equivalently, there exists a vector field $\xi\ne 0$ tangent to $M$ such that
\[
g(\xi,X)=0,\quad\mathrm{for\ all\ } X\in\Gamma(TM).
\]
Since $\dim T_pM+\dim T_pM^\perp=\dim T_p\bar M$,  the {\em radical space} $\mathrm{rad}(T_pM)=T_pM\cap T_pM^\perp$ at each $p\in M$ satisfies 
\begin{equation*}
\mathrm{rad}(T_pM)=span (\xi_p).
\end{equation*}
The {\em radical bundle} $\mathrm{rad}(TM)$ is given by
\[
\mathrm{rad}(TM)=\bigcup\limits_{p\in M}\mathrm{rad}(T_pM).
\]

A {\em screen distribution} on $M$ is a vector sub-bundle $S(TM)$ of $TM$ such that 
\begin{equation}\label{eq:descomposicion1}
TM=S(TM)\oplus_{\mathrm{orth}}  \mathrm{rad}(TM),
\end{equation}
thus, $S(TM)$ is complementary to $\mathrm{rad}(TM)$ in $TM$.

Notice that due to the Lorentzian character of $\bar{g}$, the metric restricted to $S(TM)$ is positive definite. Intuitively, an adequate choice of $S(TM)$ would enable us to relate the geometry of $(M,g)$ to a more familiar (Riemannian) geometry in $S(TM)$. Henceforth, by a {\em null hypersurface} we mean a triple $(M,g,S(TM))$.

The last main ingredient in this framework is the so called {\em transversal bundle} $\mathrm{tr}(TM)$, which is the unique rank $1$ vector bundle with the following property: given $\xi\in \Gamma (\mathrm{rad}(TM))$ there exists a unique $N\in\Gamma (\mathrm{tr}(TM))$ such that 
\begin{equation}\label{eq:N}
\bar g(\xi,N)=1,\quad \bar g(N,N)=\bar g(N,X)=0, \quad\mathrm{for\ all\ }X\in\Gamma(S(TM)).
\end{equation}
Thus, we can split the tangent bundle of the ambient manifold as
\begin{equation}\label{eq:descomposicion2}
T\bar M=TM\oplus\mathrm{tr}(TM).
\end{equation}

\begin{remark}\label{rem:gauge}
Notice that even though a choice of screen distribution $S(TM)$ determines the transversal bundle 
$\mathrm{tr}(M)$, there is an scaling gauge for the sections $\xi$ and $N$ satisfying \eqref{eq:N}. Indeed, if the smooth function $f:M\to \mathbb{R}$ never vanishes, then the smooth sections
\begin{equation}\label{eq:gauge}
\xi^\prime=f\xi,\quad N^\prime=\frac{1}{f}N
\end{equation}
also satisfy  \eqref{eq:N}.
\end{remark}

\begin{remark}\label{rem:stm}
Let us consider a vector field $\zeta$ transversal to $M$ and $\xi\in\Gamma (\mathrm{rad}(TM))$. Then the vector field $N_\zeta$ given by
\[
N_\zeta=\frac{1}{\bar{g}(\xi ,\zeta)}\left(\zeta-\frac{\bar{g}(\zeta, \zeta)}{2\bar{g}(\xi, \zeta)}\xi \right)
\]
satisfies $\bar{g}(N_\zeta,N_\zeta)=0$ and $\bar{g}(N_\zeta,\xi)=1$. Therefore, $\zeta $ induces a screen distribution $S_\zeta (TM)=span (\xi,N_\zeta)^\perp$ (refer to Eq. 2.1.7 in \cite{MR2598375})\footnote{Note that in this case $\zeta\in \Gamma (S_\zeta (TM)^\perp)$.}. The above  construction is key in the rigging approach to null hypersurface theory \cite{Gelocor2}.
\end{remark}

In virtue of decompositions \eqref{eq:descomposicion1} and \eqref{eq:descomposicion2} we can introduce the notions of induced connections, shape operators and second fundamental forms, as well as the fundamental equations relating them. 

Let $\bar\nabla$ be the Levi-Civita connection  of $(\bar{M},\bar{g})$ and $V\in\Gamma(\mathrm{tr}(TM))$. Given $X,Y\in\Gamma(TM)$, we can  write the Gauss and Weingarten formulae related to decomposition \eqref{eq:descomposicion2} as 
\begin{equation}\label{eq:GW1}
\begin{aligned}
\bar\nabla_XY&=\nabla_XY+h(X,Y),\\
\bar\nabla_XV&=-A_VX+\nabla_X^tV,
\end{aligned}
\end{equation}
where $\nabla_XY, A_VX\in\Gamma(TM)$,  and $h(X,Y),\nabla_X^tV\in\Gamma(\mathrm{tr}(TM))$. The above induces linear connections $\nabla$ and $\nabla^t$ on $TM$ and $\mathrm{tr}(TM)$, respectively. Notice however that $\nabla$ is not a metric connection, though it is torsion-free. Further, $A_V:\Gamma (TM)\to\Gamma (TM)$ and $h$ is a symmetric section of $\mathrm{Hom}^2(TM,TM;\mathrm{tr}(TM))$ that satisfies 
\begin{equation}\label{eq:nonmetricconnection}
(\nabla_X g)(Y,Z)= \bar g (h(X,Y),Z) + \bar g (Y,h(X,Z)).
\end{equation}
for all  $X,Y,Z\in\Gamma(TM)$. $\nabla$, $A_V$ and $h$ are called the {\em induced connection}, {\em shape operator} and {\em second fundamental form of} $M$, respectively.

On the other hand, relative to decomposition \eqref{eq:descomposicion1} we have another set of Gauss-Weingarten formulae for $TM$:
\begin{equation}\label{eq:GW2}
\begin{aligned}
\nabla_XPY&=\nabla_X^*PY+h^*(X,PY), \\
\nabla_XU & =-A_U^*X+\nabla_X^{*t}U,
\end{aligned}
\end{equation}
where $U\in\Gamma(\mathrm{rad}(TM))$, $P$ denotes the projection of $TM$ onto $S(TM)$; $\nabla^*$ and $\nabla^{*t}$ are connections on $S(TM)$ and $\mathrm{rad}(TM)$, respectively. Here $A_U^*:\Gamma (TM)\to \Gamma (S(TM))$ and $h^*\in\mathrm{Hom}^2(TM,S(TM),\mathrm{rad}(TM))$ are the {\em screen shape operator} and {\em screen second fundamental form} of $S(TM)$, respectively. A straightforward computation shows that $\nabla^*$ is a metric connection with respect to $\bar{g}\vert_{S(TM)}$.

For further reference, we cite some of the standard results pertaining null hypersurfaces that will be used throughout this work. For complete proofs, refer to Proposition 2.1.2 and 2.2.6 in \cite{MR1383318} and \cite{MR2598375}, respectively.

\begin{proposition}\label{prop:general}
Let $(M,g,S(TM))$ be a null hypersurface, $U\in\Gamma(\mathrm{rad}(TM))$ and  $V\in\Gamma(\mathrm{tr}(TM))$. Then
\begin{enumerate}
\item $A_V$ is $S(TM)$ valued.
\item $A^*_UU=0$.
\item $A^*_\xi$ is symmetric with respect to $g$.
\item  The following are equivalent:
\begin{itemize}
\item $S(TM)$ is an integrable distribution.
\item $h^*$ is symmetric on $S(TM)$.
\item $A_V$ is symmetric respect to $g\vert_{S(TM)}$.
\end{itemize}
\end{enumerate}
\end{proposition}

For $Y\in \Gamma (T\bar{M})$, applying decompositions \eqref{eq:descomposicion1} and \eqref{eq:descomposicion2} leads to 
\begin{equation}\label{eq:ZXN}
Y = \overset { * } { Y } + Y _ \xi + Y _ N 
\end{equation}
where 
\[
Y _ \xi = \overline { g } ( Y , N ) \xi ,\quad   Y_ N = \overline { g } ( Y , \xi ) N
\]
and $ \overset { * } { Z } $ is the projection of $Y$ on $ S ( T M ) $. We now establish the basic relations arising from coupling the above relation to the Gauss and Weingarten formulae when $X\in\Gamma (S(TM))$.
\begin{eqnarray*}
\bar{\nabla}_XY &=& \bar{\nabla}_X(\overset{*}{Y}+Y_\xi )+\bar{\nabla}_XY_N\\
&=& \nabla_X(\overset{*}{Y}+Y_\xi )+h(X, \overset{*}{Y}+Y_\xi )-A_{Y_N}X+\nabla_X^tY_N\\
&=& 
 \nabla_X\overset{*}{Y}+\nabla_X Y_\xi +h(X, \overset{*}{Y}+Y_\xi )-A_{Y_N}X+\nabla_X^tY_N\\
&=& \overset{*}{\nabla}_X\overset{*}{Y}+\overset{*}{h}(X,\overset{*}{Y})-A_{Y_\xi}^*X+\nabla_X^{*t}Y_\xi+h(X, \overset{*}{Y} )-A_{Y_N}X+\nabla_X^tY_N.
\end{eqnarray*}
In summary
\begin{equation}\label{eq:components}
\bar{\nabla}_XY = \underbrace{\overset{*}{\nabla}_X\overset{*}{Y}-A_{Y_\xi}^*X-A_{Y_N}X}_{\in\Gamma (S(TM))}+\underbrace{\overset{*}{h}(X,\overset{*}{Y})+\nabla_X^{*t}Y_\xi }_{\in\Gamma (\mathrm{rad}(TM))}+\underbrace{h(X, \overset{*}{Y})+\nabla_X^tY_N}_{\in\Gamma (\mathrm{tr}(TM))}
\end{equation}

\begin{remark}\label{rem:LC}
Notice that when $S(TM)$ is integrable their leaves form a foliation of $M$ by Riemannian manifolds $(\widehat{M},\hat{g})$. As a consequence of Frobenius theorem, in such cases $\overset{*}{\nabla}$ is torsion free and hence it coincides with the Levi-Civita connection $\widehat{\nabla}$ of each leaf $(\widehat{M},\widehat{g})$. 
Moreover, consider the Gauss-Weingarten formulae associated to decomposition $T\bar{M}=S(TM)\oplus S(TM)^\perp$, that is,
\begin{equation}\label{eq:codim2}
	\bar{ \nabla } _ X Y  = \bar{ \nabla } _ X ( \overset { * } { Y } + Y ^ \perp ) 
	 = \widehat { \nabla } _ X \overset { * } { Y } + \alpha ( X , \overset { * } { Y } ) - \widehat { A } _ { Y ^ \perp }  X  + \nabla _ X ^ \perp Y ^ \perp ,
	\end{equation}
where $\alpha$, $\widehat { A } _ { Y ^ \perp }$, $\nabla _ X ^ \perp Y ^ \perp$ are the second fundamental form, shape operator and normal connections of $\widehat{M}$ as a codimension two submanifold immersed in $\bar{M}$ (see \cite{DajczerTojeiro}). Thus, comparing the $S(TM)$ components in Eqs. \eqref{eq:components} and \eqref{eq:codim2} leads to 
\begin{equation}\label{eq:shapeops}
\widehat { A } _ { Y ^ \perp }  X =A_{Y_\xi} X+A_{Y_N}X, \qquad X\in\Gamma (S(TM)).
\end{equation}
\end{remark}

Let us now consider a null frame $\{\xi , N\}$ as in Eq. \eqref{eq:N} and denote by $B$, $C$  the bilinear forms associated to $h$ and $h^*$, respectively. That is 
\begin{equation}\label{eq:h&h}
h(X,Y) = B(X,Y)N, \qquad  h^*(X,PY) = C(X,PY)\xi . 
\end{equation}
Thus, in virtue of equations \eqref{eq:GW1} and  \eqref{eq:GW2}   we have
\begin{equation}\label{eq:relacionAh}
\begin{aligned}
B(X,Y) &=\bar{g}(h(X,Y),\xi)= g(Y,A_\xi^*X),\\ 
C(X,Y)&= \bar{g} (h^*(X,PY),N) =g(PY,A_NX),
\end{aligned}
\end{equation}
for all $X,Y\in\Gamma (TM)$. 
It then follows 
that
\begin{equation}\label{eq:bcero}
B(X,\xi )=0\quad \forall X\in\Gamma (TM).
\end{equation}
Notice  that since $h$ is symmetric, then by Eq. \eqref{eq:relacionAh} $A_N$ is symmetric with respect to $g$. Hence Proposition \ref{prop:general} yields that  the same is true for $A_\xi$ provided $S(TM)$ is integrable.

The {\em null mean curvature} $H$ is the trace of the shape operator $A^*_\xi$. Due to Eq. \eqref{eq:bcero} we have
\begin{equation*}
H=\sum_{i=1}^nB(e_i,e_i),
\end{equation*}
where $\{e_1,\ldots ,e_n\}$ denotes an orthonormal frame field in $S(TM)$.

Furthermore, note that $\bar{g}(\nabla^t_XN,\xi )=-\bar{g}(\nabla_X^{*t}\xi ,N)$. Hence, in terms of the one form $\tau$ defined by
\begin{equation}\label{eq:tau}
\tau (X)= \bar{g}(\nabla^t_XN,\xi )
\end{equation}
the Weingarten formulae read as
\begin{equation}\label{eq:GWapp}
\bar\nabla_XV=-A_VX+\tau (X)V,\quad \nabla_XU =-A_U^*X-\tau (X)U .
\end{equation}

\begin{remark}\label{rem:tauzero}
Notice that when $\tau \equiv 0$ Eq. \eqref{eq:GWapp} looks just like the standard semi-Riemannian (non-degenerate) Weingarten  formula for a unit normal vector field. In that sense, the vanishing of $\tau$ can be interpreted as a normalization that ensures the closest possible analogue to the classical submanifold theory. As it turns out, a choice of $\{\xi ,N\}$ that yields $\tau\equiv 0$ is always possible if the  Ricci tensor $Ric^\nabla$  vanishes (see Prop. 2.4.2 in \cite{MR2598375}). 
\end{remark}

The curvature operators $\bar{R}$, $R$, $R^*$ associated to the linear connections $\bar{\nabla}$, $\nabla$, $\nabla^*$ are intertwined in a series of fundamental equations that resemble the standard Gauss-Codazzi structure equations in the non-degenerate case. We only present one of them, that will be used in Theorem \ref{theoCMC} and refer to it as the {\em Codazzi equation of} $M$. For all $X,Y,Z\in\Gamma (TM)$ and $U\in\Gamma (\textrm{rad}(TM))$ we have
\[
\bar{g}(\bar{R}(X,Y)Z, U) =\bar{g}((\nabla_Xh)(Y,Z)- (\nabla_Yh)(X,Z),U) 
\]
Thus, once the null fields $\{\xi , N\}$ are fixed,
\begin{equation}\label{eq:Codazzi}
\bar{g}(\bar{R}(X,Y)Z, \xi) =(\nabla_XB)(Y,Z)- (\nabla_YB)(X,Z)+\tau (X)B(Y,Z)-\tau (Y)B(X,Z)
\end{equation}
where Eq. \ref{eq:nonmetricconnection} translates to
\[
(\nabla_XB)(Y,Z)=X\cdot B(Y,Z)-B(\nabla_XY,Z)-B(Y,\nabla_XZ).
\]

Many of the most significant geometrical restrictions widely studied in the semi-Riemannian (non degenerate) theory of submanifolds have analogues in the null setting. For instance, $M$ is {\em totally umbilical} (resp. {\em totally geodesic}) if there exists a smooth function $\lambda$ such that $h(X,Y)=\lambda g(X,Y)$ (resp. $h\equiv 0$). Similarly, $S(TM)$ is {\em totally umbilical} (resp. {\em totally geodesic}) if there exists a smooth function $\lambda$ such that $h^*(X,Y)=\lambda g(X,Y)$ (resp. $h^*\equiv 0$). In terms of the shape operators, we have that $M$ is totally umbilical (resp. totally geodesic) if $A_\xi^* PX = \lambda PX$ ($A_\xi^*\equiv 0$) and likewise, $S(TM)$ is {\em totally umbilical} (resp. {\em totally geodesic}) if $A_NX=\lambda PX$ ($A_N\equiv 0$) for all $X\in\Gamma (TM)$. 
Notice, that all these conditions, as well as the vanishing of the null mean curvature ($H\equiv 0$) do not depend on the choice of the null frame $\{\xi ,N\}$.

As an illustrative example, let us focus on the case of {\em generalized Robertson-Walker} spacetimes, or GRW spacetimes for short. Let us recall that they are defined as Lorentzian warped products of the form $\bar{M}=-I\times_{\varrho} F$, with $(F,g_F)$ a Riemannian manifold and $\varrho$ a positive smooth function defined on the real interval $I$. Thus we have
\begin{equation}\label{eq:GRW}
\bar g=-\sigma^*(dt^2)+(\varrho\circ\pi)^2 \pi^*(g_F).
\end{equation}
where $\sigma$ and $\pi$ are the projections of $\bar M$ over $I$ and $F$, respectively. If the fiber $F$ is a Riemannian manifold of constant sectional curvature suitable choices of the warping function deliver open sets of all Lorentzian spaceforms as follows:
\begin{itemize}
\item Lorentz-Minkowski space: $\mathbb L^{n+2}=-\mathbb R\times_{\text{id}}\mathbb R^{n+1}$.
\item de Sitter space: $\mathbb S_1^{n+2}=-\mathbb R\times_{\cosh}\mathbb S^{n+1}$.
\item Anti de Sitter space $\mathbb H_1^{n+2}=-\left(-\frac{\pi}{2},\frac{\pi}{2}\right)\times_{\cos}\mathbb H^{n+1}$.
\end{itemize}

\begin{example}\label{ex:GRW}
Let $f:F\to\mathbb R$  be a {\em transnormal function}, that is, a smooth function such that $\vert \text{grad}\,f\vert =\varrho\circ f$. Thus the graph $M=\{\,(f(p),p)\,\vert\,p\in F\,\}$ is a degenerate hypersurface in the GRW spacetime $\bar M=-I\times_{\varrho} F$. Further, let $S^*(TM)$  be the family of  tangent bundles of the level hypersurfaces $S_t=M\cap\left(\{t\}\times F\right)$. Consider now $\xi\in \Gamma (\mathrm{rad}(TM))$ and $N\in \Gamma (\mathrm{tr}(TM))$ given by
\begin{equation*}
\xi=\frac{1}{\sqrt{2}}\left(\partial_t+\frac{1}{(\varrho\circ f)^2} \overline{\text{grad} f}\right),\quad N=\frac{1}{\sqrt{2}}\left(-\partial_t+\frac{1}{(\varrho\circ f)^2} \overline{\text{grad} f}\right) ,
\end{equation*}
where $\overline{\text{grad} f}$ denotes the lift of $\text{grad} f$ to $\Gamma (T\bar{M})$. In this setting we have that the shape operators of $(\bar{M},g,S^*(TM))$ satisfy
\begin{equation}\label{eq:sqcRW}
\frac{1}{\sqrt{2}}(A_N-A_\xi^*)=\frac{\varrho'}{\varrho}P.
\end{equation}
\end{example}

The above example serves as motivation for the concept of screen quasi-conformal hypersurface, which will be studied in the next section. For a detailed account, see \cite{MR3874677,MR4135826}.

\subsection{Closed conformal vector fields}

A property that characterizes GRW spacetimes relates to the notion of closed conformal vector field.  Let us recall that $Z\in\Gamma (T\bar{M})$ is {\em closed conformal} ---or CC for short--- if there exists a smooth function $\varphi :\bar{M}\to\mathbb{R}$ such that
\begin{equation}\label{eq:ccvf} 
\bar{\nabla}_X{Z}=\varphi X
\end{equation}
for all $X\in\Gamma (T\bar{M})$.  Indeed, GRW spacetimes can also be described as those Lorentzian manifolds having a timelike CC vector field \cite{chen04}. Throughout this work, $\varphi$ will always denote the function associated to a CC vector field.

In the case of semi-Riemannian spaceforms, when viewed as hyperquadrics immersed in a semi-Euclidean space, CC vector fields arise as projections of parallel vector fields defined on the respective semi-Euclidean space \cite{NRS01,GG}. Notice that according to Eq. \eqref{eq:GRW}, the vector field $Z=\varrho \partial_t$ of Example \ref{ex:GRW} is closed conformal. 

Other  important classes of CC vector fields include {\em parallel} ($\varphi \equiv 0$), {\em homothetic} ($\varphi$ is a constant function) and {\em radial}, that is, when $\bar{M}$ is a semi-Euclidean space and $Z(x)=\varphi x+Z_0$, with $Z_0$ a parallel vector field. 

One of the most remarkable features of CC vector fields is that their length $\vert Z\vert  =\sqrt{\vert \bar{g}(Z,Z) \vert}$ is constant (if  not null) along directions orthogonal to them (see Lemma 2.7 in \cite{NRS01}).

\begin{lemma}\label{prop:ccc}
Let $Z\in \Gamma (T\bar{M})$ be a CC vector field and $X\in \Gamma (T\bar{M})$ such that $\bar{g}(Z,X)=0$. If $\bar{g}(Z,Z)$ never vanishes, then $X\cdot\vert Z\vert =0$.
\end{lemma}
\begin{proof}
Notice that
\begin{equation*}
X\cdot\bar{g}(Z,Z)=2\bar{g}(\bar{\nabla}_XZ,Z)=2\varphi\bar{g}(X,Z)=0. 
\end{equation*}
The result follows at once.
\end{proof}

In very recent times, we have witnessed an increasing interest in exploring the interplay between a CC vector field and the geometry of certain classes of null hypersurfaces when the ambient manifold is a spaceform  \cite{Samuel}, a space of quasi constant curvature \cite{CiriacoConforme}, or in relation to the Raychaudhuri equation \cite{CiriacoBenji}. A common starting point consists in describing a CC vector field $Z$ in terms of decomposition  \eqref{eq:ZXN}, that is, 
$Z = \overset { * } { Z } + Z _ \xi + Z _ N$, where $ Z _ \xi = \overline { g } ( Z , N ) \xi $, $ Z _ N = \overline { g } ( Z , \xi ) N $ and $ \overset { * } { Z }\in\Gamma (S(TM)) $. We now establish the basic relations arising from coupling the above relation to the Gauss and Weingarten formulae. Compare the following result with Proposition 3.1 in \cite{CiriacoConforme} and Proposition 3.7 in \cite{Samuel}.

\begin{lemma}\label{eqcomponents}
Let $ ( M ,g, S ( T M ) ) $ be a null hypersurface of $ (\bar{M},\bar{g}) $. If $ Z\in\Gamma(T\bar{M}) $ is a CC vector field, then for any $ X \in\Gamma ( S ( T M ) )$:
\begin{enumerate}
	\item[(a)] $ \varphi X = \overset { * }{ \nabla } _ X \overset { * }{ Z } - \overline{ g }( Z , N ) A ^ * _ \xi  X  - g ( Z , \xi ) A _ N  X  , $
	\item[(b)] $ C ( X , \overset { * } { Z } ) +  X \cdot g ( Z , N ) - g ( Z , N ) \tau ( X ) = 0 $ and
	\item[(c)] $ B ( X , \overset { * } { Z } ) +  X \cdot g ( Z , \xi  ) + g ( Z , \xi ) \tau ( X ) = 0 $.
\end{enumerate}
\end{lemma}

\begin{proof}
	By Eq. \ref{eq:components} we have
	\begin{eqnarray*}
	\varphi X = (\overset{*}{\nabla}_X\overset{*}{Z}-A_{Z_\xi}^*X-A_{Z_N}X) + (\overset{*}{h}(X,\overset{*}{Z}+\nabla_X^{*t}Z_\xi )+( h(X, \overset{*}{Z})+\nabla_X^tZ_N)
	\end{eqnarray*}
	 Thus comparing the $S(TM)$ components in both sides of the above equation yields
\[
\varphi X  = \overset { * } { \nabla } _ X  \overset { * } { Z } - g ( Z , N ) A _ \xi ^ *  X  - g ( Z , \xi ) A _ N  X
\]
so (a) holds. Similarly, by looking at the $\mathrm{rad}(TM)$ and   $ \mathrm { tr } ( M ) $ components  we get
\begin{eqnarray*}
	0 &=& h ( X , \overset { * } { Z } ) + \nabla _ X ^ t Z _ N, \\
	0 & =& \overset { * } { h } ( X , \overset { * } { Z } ) + \overset { * } { \nabla } _ X ^ t  Z _ \xi  .
	\end{eqnarray*}
Recall that $h(X,\overset { * } { Z }) = B(X, \overset { * } { Z }) N $ and thus
	\[
	\nabla _ X ^ t Z _ N = \nabla _ X ^ t g(Z, \xi) N =(( X \cdot g ( Z , \xi ) ) + g ( Z , \xi ) \tau ( X ))N.
	\]
so (c) holds as well. Finally, (b) follows from an argument analogous to the proof of (c) above.
\end{proof}

As an immediate consequence of Lemma \ref{eqcomponents} (a) we can
describe the cases in which $Z$ lies entirely on $\mathrm{rad}(TM)$ or $\mathrm{tr}(TM)$ (compare to Thrm. 3.9 in \cite{Samuel}).

\begin{corollary}\label{prop:totumb}
Let $ ( M ,g, S ( T M ) ) $ be a null hypersurface of $ (\bar{M},\bar{g}) $ and  $ Z\in\Gamma(T\bar{M}) $ a CC vector field. Then
\begin{enumerate} 
\item If  $ \overset { * } { Z } = Z _ \xi = 0  $ , then $ M $ is totally umbilical.
\item If $ \overset { * } { Z } = Z _ N = 0 $, then $ S ( T M ) $ is totally umbilical.
\end{enumerate}
Moreover, if $Z$ is parallel then
\begin{enumerate} 
\item If  $ \overset { * } { Z } = Z _ \xi = 0  $ , then $ M $ is totally geodesic.
\item If $ \overset { * } { Z } = Z _ N = 0 $, then $ S ( T M ) $ is totally geodesic.
\end{enumerate}
\end{corollary}

The following Lemma will be key in several results in Sections \ref{S1} y \ref{S2}.

\begin{lemma}\label{lemma:zperp}
Let $ ( M ,g, S ( T M ) ) $ be a null hypersurface of $ (\bar{M},\bar{g}) $ and  $ Z\in\Gamma(T\bar{M}) $ a CC vector field. If $S(TM)$ in integrable, then
\[
X\cdot (\bar{g}(X,\xi )\bar{g}(X,N)) =-\bar{g} (X,A_{Z^\perp}\overset{*}{Z}), \quad X\in\Gamma (S(TM)).
\]
\end{lemma}

\begin{proof}
Since $S(TM)$ is integrable then $A_N$ is symmetric. Thus by points (b) and (c) of Lemma \ref{eqcomponents} and Eq \eqref{eq:shapeops} we have
\begin{align*}
X\cdot (\bar{g}(X,\xi )\bar{g}(X,N)) &= -\bar{g}(Z,\xi)\bar{g}(X,A_N\overset{*}{Z})-\bar{g}(Z,N)\bar{g}(X,A^*_\xi \overset{*}{Z})\\
&= -\bar{g}(X,A_{Z_N}\overset{*}{Z})-\bar{g}(X,A^*_{Z_\xi} \overset{*}{Z})\\
&= -\bar{g}(X,A_{Z^\perp}\overset{*}{Z}).
\end{align*}
\end{proof}

We end up this section with the computation of the screen gradient of the norm of a CC vector field $Z$.

\begin{lemma}\label{eqGrads}
Let $ ( M ,g, S ( T M ) ) $ be a null hypersurface of $ (\bar{M},\bar{g}) $ and $ Z\in\Gamma(T\bar{M}) $ a CC vector field. Then 
\[
    \overset { * } { \nabla } \vert Z \vert  = \frac{\varepsilon_Z\varphi}{\vert Z\vert} \overset { * } { Z },
\]
where $\varepsilon_Z=\pm 1$ is the sign of $\overline{g}(Z,Z)$.
\end{lemma}
\begin{proof}
	Given an orthonormal frame $ \{ e _ 1, e _ 2 , \ldots , e _ n \} $ of $ S ( T M ) $  we have
	\begin{align*}
	\overset { * } { \nabla } \bar{ g }  ( Z , Z  ) & = \sum _ { i = 1 } ^ n   ( e _ i \cdot \bar{ g } ( Z , Z )  ) e _ i= \sum _ { i = 1 } ^ n  2 \bar{ g }  ( \overline { \nabla } _ { e _ i } Z , Z  ) e _ i= \sum _ { i = 1 } ^ n  2 \bar{ g } ( \varphi e _ i , Z ) e _ i \\
	& = 2 \varphi \overset { * } { Z }.
	\end{align*}
Finally
\[ \overset { * } { \nabla } \lvert Z \rvert = \frac { \overset { * } { \nabla } \varepsilon _ Z g ( Z , Z ) } { 2 \lvert Z \rvert } = \frac { \varepsilon _ Z \varphi } { \lvert Z \rvert } \overset { * } { Z }.  \]
\end{proof}

\section{Screen quasi-conformal null hypersurfaces} 

One key question that arises naturally is whether the geometry of $S(TM)$ as a subbundle of $TM$ is related in a natural way to its geometry as a subbundle of $T\bar{M}$. In \cite{MR2039644} K. Duggal and C. Atindogbe define the notion of screen conformal null hypersurface and show that the answer is affirmative for this class of hypersurfaces. In precise terms, a null  hypersurface $(M,g,S(TM))$ is {\em screen conformal} if its shape operators are linearly dependent, so there exists $\phi \in\mathcal{C}^\infty (M)$ such that 
\[
A_N=\phi A_\xi^*.
\]
Numerous classification results exist for such hypersurfaces when the ambient manifold $(\bar{M},\bar{g})$ has constant curvature \cite{MR2598375}. However, in virtue of Example \ref{ex:GRW} an important class of space-times does not fit in this context. Thus, in order to incorporate such examples, the notion of screen quasi-conformal null hypersurface was introduced in \cite{MR4135826}.

\begin{definition}\label{def:qconforme}
A null hypersurface $(M,g,S(TM))$ of $(\bar{M},\bar{g})$ is {\em screen quasi-conformal} if the shape operators $A_N$ and $A_\xi^*$ of $M$ and $S(TM)$ satisfy
\begin{equation*}
A_N=\phi A_\xi^*+ \psi P,
\end{equation*}
for some functions $\phi,\psi\in\mathcal{C}^\infty (M)$; where $P:\Gamma(TM)\to \Gamma(S(TM))$ is the natural projection. We call $(\phi,\psi)$ a {\em quasi-conformal pair}. 
\end{definition}

Screen quasi-conformal null hypersurfaces satisfying classical geometric restrictions (totally geodesic, totally umbilical, isoparametric, Einstein) have been studied recently (see \cite{Gelocor} and references therein). Moreover, if $(M,g,S(TM))$ is screen quasi-conformal, then $S(TM)$ is integrable (see Theorem 3.7 in \cite{MR4135826}). As can be readily checked, the null hypersurfaces $(M,g,S^*(TM))$ depicted in Example  \ref{ex:GRW} are screen quasi-conformal with quasi-conformal pair $(1,\sqrt{2}\varrho'/\varrho )$. \footnote{Notice that in this case the following additional conditions hold: (a) $\phi$ and $\psi$ are constant along the leaves of $S(TM)$, and (b) $\tau (X)=0$ for all $X\in\Gamma (S(TM))$. }

At this point it is natural to ask under which circumstances we can  endow a degenerate hyperfurface with a screen distribution in a way that the corresponding null hyperfurface $(M,g,S(TM))$ is screen conformal. In \cite{MR2598375},  Theorem 2.3.5 establishes that if the ambient space admits a \textit{parallel} vector field, then it induces a screen conformal structure in any null hyperfurface. As the next result shows, there is an analogous connection between closed and conformal vector fields and screen quasi-conformal distributions, thus generalizing the aforementioned result.

\begin{theorem}\label{theoscreenquasi}
Let $(M,g)$ be a degenerate hypersurface immersed  $(\bar{M},\bar{g})$ and $Z\in \Gamma (T\bar{M})$ a closed and conformal vector field  with $\bar{g} (Z,Z)\neq 0$. If $Z$ is nowhere tangent to $M$ then there exists a screen distribution $S^\prime (TM)$ so that $(M,g,S^\prime (TM))$ is screen quasi-conformal. Furthermore, $\tau (X)=0$ $\forall X\in\Gamma (S^\prime(TM))$.
\end{theorem}

\begin{proof} 
Since $Z_p\not\in T_pM$ we have that $E=span (Z)\oplus \mathrm{Rad}(TM)$ is a rank $2$ vector bundle in which $\bar{g}\vert_E$ has Lorentzian character. Consider $S^\prime (TM)=E^\perp$ and let $\mathrm{tr}(TM)$ be a null line bundle complementary to $\mathrm{Rad}(TM)$ in $E$. Since $Z\in\Gamma (E)$ we have $Z_r\in\Gamma (\mathrm{Rad}(TM))$ and  $Z_t\in\Gamma (\mathrm{tr}(TM))$ such that 
\[
Z=Z_r+Z_t
\]
Hence, we can choose $\xi\in\Gamma (\mathrm{Rad}(TM))$ and $N\in\Gamma (\mathrm{Rad}(TM))$ as
\[
\xi =Z_r,\quad N=\frac{1}{\bar{g}(Z_r,Z_t)}Z_t,
\]
or equivalently
\[
Z=\xi + \theta N,  
\]
where
\[
\theta =\bar{g}(Z_r,Z_t)=\frac{1}{2}\bar{g}(Z,Z).
\]
Thus, for $X\in \Gamma (TM)$ we have
\begin{eqnarray*}
\varphi X &=& \bar{\nabla}_XZ\\
&=& \bar{\nabla}_X\xi +(X\cdot \theta )N+\theta \bar{\nabla}_XN\\
&=& \nabla_X\xi + h(X,\xi )+(X\cdot \theta )N+\theta (-A_NX+\tau (X)N)\\
&=& -A^*_\xi X-\theta A_NX-\tau (X)\xi +(X\cdot \theta )N+\theta \tau (X)N.
\end{eqnarray*}
Thus we immediately have
\begin{eqnarray*}
\varphi PX &=& -A^*_\xi X-\theta A_NX,\\
0 &=& (X\cdot \theta )N +\theta \tau (X)N.
\end{eqnarray*}
Hence, the first equation implies that $(M,g,S^\prime (TM))$ is screen quasi-conformal. Moreover, by Lemma \ref{prop:ccc}, $\theta \neq 0$ is constant along $S^\prime (TM)$ and the second equation implies $\tau (X)=0$.
\end{proof}

Finally, let us notice that if the CC field is orthogonal to the screen distribution, then $(M,g,S(TM))$ is screen quasi-conformal.

\begin{proposition}
Let $ ( M ,g, S ( T M ) ) $ be a null hypersurface of  $ (\bar{M},\bar{g}) $. If $ Z\in\Gamma(T\bar{M}) $ is a CC vector field such that  $ \overset { * } { Z } = 0 $, then $ ( M ,g, S ( T M ) ) $ is screen quasi-conformal.
\end{proposition}
\begin{proof}
	By the same calculations as in Theorem \ref{theoscreenquasi} we get
	\begin{align*}
	\varphi X & = - g ( Z , N ) A ^ * _ \xi X  - g ( Z , \xi ) A _ N X ,
	\end{align*}
	so we obtain
	\begin{align*}
	A _ N  X = - \frac { g ( Z , N ) } { g ( Z , \xi ) } A ^ * _ \xi X  + \frac { \varphi }{ g ( Z , \xi ) } X .
	\end{align*}
\end{proof}

\section{Constant angle null hypersurfaces}\label{S1}

As it is often the case, dealing with degenerate submanifolds requires some adaptation of the standard  definitions of classical geometrical concepts. For the case of a semi-Riemannian hypersurface, the notion of angle respect to a vector field $V$, as given by Eq. \eqref{eq:angle}, heavily depends upon the normalization of a vector field orthogonal to the hypersurface, thus providing a unit normal vector field $\mathbf{n}$ along it. For a null hypersurface $(M,g,S(TM))$ the absence of a canonical normalization in the null directions $\xi$ and $N$ rules out the possibility for the functions
\[
\measuredangle (V,\xi ):=\frac{\bar{g}(V , \xi )}{\vert V\vert}, \quad \measuredangle (V,N):= \frac{\bar{g}(V, N)}{\vert V\vert},
\]
to be suitable candidates for describing an angle between $M$ and a nowhere null vector field $V$.

Indeed, the gauge freedom of Eq. \eqref{eq:gauge} enables that such functions (when non vanishing) might take any possible value after a scaling. In particular, we can always choose $f$ such that $\measuredangle (V,\hat{\xi})$ is constant. For instance, by taking $f=\bar{g}(V ,\xi )/\vert Z\vert $ we get $\measuredangle (V,\hat{\xi}) =1$. A similar situation holds for $\measuredangle (V,\hat{N})$. Notice however, that in general we can not use Eq. \eqref{eq:gauge} in order to make {\em both} functions $\measuredangle (V,\hat{\xi})$ and $\measuredangle (V,\hat{N})$ simultaneously constant. Thus, in the spirit of keeping the algebraic form of Eq. \eqref{eq:angle}, we consider the following definition.

\begin{definition}\label{const angle}
Let $ ( M ,g, S ( T M ) ) $ be a null hypersurface of $( \bar{ M },\bar{g}) $ and $ V $ be a nowhere null vector field along $M$. We say that $( M ,g, S ( T M ) ) $ has constant angle respect to $ V $ if there exists $\xi\in\Gamma (\textrm{rad}(TM))$ and $N\in\Gamma (\textrm{tr}(TM))$ such that the functions 
\begin{equation}\label{eq:angf}
\measuredangle (V,\xi )= \frac{ \bar{ g }(  V , \xi )}{ \vert V \vert }\quad  \textrm{and}\quad  \measuredangle (V,N )=\frac{ \bar{ g }(  V  , N )}{ \vert V \vert }
 \end{equation}
 are constant.
\end{definition}

At first sight, one could object that Definition \ref{const angle} is not very useful in concrete examples, as the non-constancy of the angle functions \eqref{eq:angf} for a given pair of null vector fields $\{\xi ,N\}$ is indecisive as whether $(M,g,S(TM))$ has constant angle respect to $V$ or not. However, as the next result shows, we can formulate Definition \ref{const angle} in terms that do not depend on any particular choice of a scaling.  

\begin{lemma}\label{ca:easy}
Let $ ( M ,g, S ( T M ) ) $ be a null hypersurface in $ (\bar{ M },\bar{g}) $ and $ V $ a nowhere null vector field along $M$. Then the following statements are equivalent:
\begin{enumerate}
    \item\label{c01} $ ( M ,g, S ( T M ) ) $ has constant angle with respect to $Z$.
    \item\label{c02} For any choice of $\xi\in \Gamma (\textrm{rad}(M))$ and $N\in \Gamma (\textrm{tr}(M))$ the function $\measuredangle (V ,\xi )\measuredangle (V,N)$ is constant.
\item\label{c03} $g(V^*/\vert V\vert ,V^*/\vert V\vert)$ is constant.
\end{enumerate}
\end{lemma}

\begin{proof}
We first show the equivalence between conditions (\ref{c01}) and (\ref{c02}). Let us assume $\measuredangle (V, \xi )$ and $\measuredangle (V,N)$ are constant. Then for $\hat{\xi}=f\xi$ and $\hat{N}=N/f $ we have
that $\measuredangle (V,\hat{\xi} )\measuredangle (V,\hat{N})=\measuredangle (V,\xi )\measuredangle (V,N)$, so it is clearly a constant function.

Conversely, assume $k=\measuredangle (V,\xi )\measuredangle (V,N)$ is constant. Take 
\[
f=\frac{\measuredangle (V,N)}{k},
\]
hence
\[
\measuredangle (V,\hat{\xi}) = f\measuredangle (V,\xi )=1,\quad
\measuredangle (V,\hat{N}) =  \measuredangle (V,N)/f=k.
\]
To establish the equivalence between conditions (\ref{c02}) and (\ref{c03}) notice that 
\begin{equation}\label{eq:prods}
g(V,V)=g(V^*,V^*)+2g(V,\xi )g(V,N),
\end{equation}
thus
\[
g\left(\frac{V^*}{\vert V\vert} , \frac{V^*}{\vert V\vert}\right) = \frac{g(V,V)}{\vert V\vert^2}- 2\frac{g(V,\xi )}{\vert V\vert}\frac{g(V,N)}{\vert V\vert}
= \varepsilon_Z-2\measuredangle (V,\xi)\measuredangle (V, N)
\]
and the result follows immediately.

\end{proof}

Recall that in the semi-Riemannian setting, constant angle hypersurfaces can be defined via the squared norm of  $V^\top /\vert V\vert$, or equivalently, the squared norm of $V^\perp /\vert V\vert$; where $V^\top$ and $V^\perp$ denote the tangent and normal components of $V$.  Thus Lemma \ref{ca:easy} can be interpreted as a generalization to the null context of this classical result (see for instance Lemma 3.3 in \cite{NRS01}).

Furthermore, two exceptional cases arise naturally when dealing with non degenerate hypersurfaces $(M,g)$, namely, when the vector field $V$ is orthogonal to the hypersurface and when it is tangent to it. For instance, if the normal distribution to $V$ is integrable, then a unit normal vector field to one of its integral manifolds in necessarily collinear to $V$, and thus we can think of the hypersurface as making a zero angle with respect to $V$. This is for instance the case of hyperplanes or hyperquadrics in semi-Euclidean spaces when we consider parallel or radial vector fields\footnote{Recall that CC vector fields in semi-Euclidean spaces are necesarily either parallel or radial.}. On the other hand, if $V$ is tangent to the hypersurface $(M,g)$ then $V$ is orthogonal to any unit normal vector field, as it is indeed the case for cylinders ($V$ parallel) and cones ($V$ radial) in semi-Euclidean spaces.

\begin{example}\label{example:trivial}
In view of Lemma \ref{ca:easy} we have that the proper analogue for a null hypersurface $(M,g,S(TM))$ making a zero angle with respect to a vector field $V$ are precisely those for which $\overset{*}{V}=0$, that is, when $V\in \Gamma (S(TM)^\perp)$.  Notice that if we assume $V$ to be a CC vector field, then Proposition \ref{prop:totumb} ensures that  totally geodesic or totally umbilic null hypersurfaces fit in this description, in close analogy to the non degenerate case. Moreover, if we take a vector field $\zeta \in\Gamma (T\bar{M})$ transversal to $M$ and consider the associated screen distribution $S_\zeta(TM)$ as in Remark \ref{rem:stm}, then $\overset{*}{\zeta}=0$ and consequently $(M,g,S_\zeta (TM))$ makes a constant angle with respect to $\zeta$.
\end{example}

\begin{example}\label{example:caMink}
Let us consider a null hyperplane $\Pi\subset\mathbb{R}^{n+2}_1$ and a parallel vector field $V\in\Gamma (T\mathbb{R}^{n+2}_1)$. We can always choose $\xi$ to be parallel along $\Pi$, and thus $\measuredangle (V,\xi )$ is constant. Moreover, if we take $S^\prime (T\Pi )$ as the tangent spaces to a family of parallel planar sections of $\Pi$ then $N$ is also parallel along $M$ and thus $\measuredangle (V,N)$ is constant as well. Consequently, $(\Pi ,g, S^\prime (T\Pi ))$ is a constant angle null hypersurface of $\mathbb{R}^{n+2}_1$.

Conversely, assume $(\Pi ,g,S(T\Pi ))$ is a constant angle null hypersurface with respect to a parallel vector field $V$. Further, without loss of generality we can assume $V =e_0$ if $V$ is timelike, choose $\xi\in \Gamma (\textrm{rad} (TM))$ such that $\measuredangle (V,\xi )=\bar{g}(V,\xi )=1$.   Let $N=(N_0,N_1,\ldots , N_{n+1})$ such that $\measuredangle (N,V)= -N_0$ is constant. Thus $N_2^2+\cdots +N_{n+1}^2=1$ and 
\[
S(TM)=span \{e_2-N_2\xi ,\ldots , e_{n+1}-N_{n+1}\xi \}.
\]
Thus, for $n=1$ we have that $(\Pi ,g,S(TM))$ is constant angle hypersurface only if $S(TM)$ is planar, that is, $S(TM)_p$ is generated by the intersection of $T_pM$ and a hyperplane in the ambient space $\mathbb{R}^{m+2}_1$ passing through $p$ (see Figure \ref{fig:plano}). On the other hand, for $n\ge 2$ we have infinitely many non planar null  hypersurfaces.
\end{example}

\begin{figure}[ht!]\label{fig:plano}
\centering \includegraphics[scale=0.25]{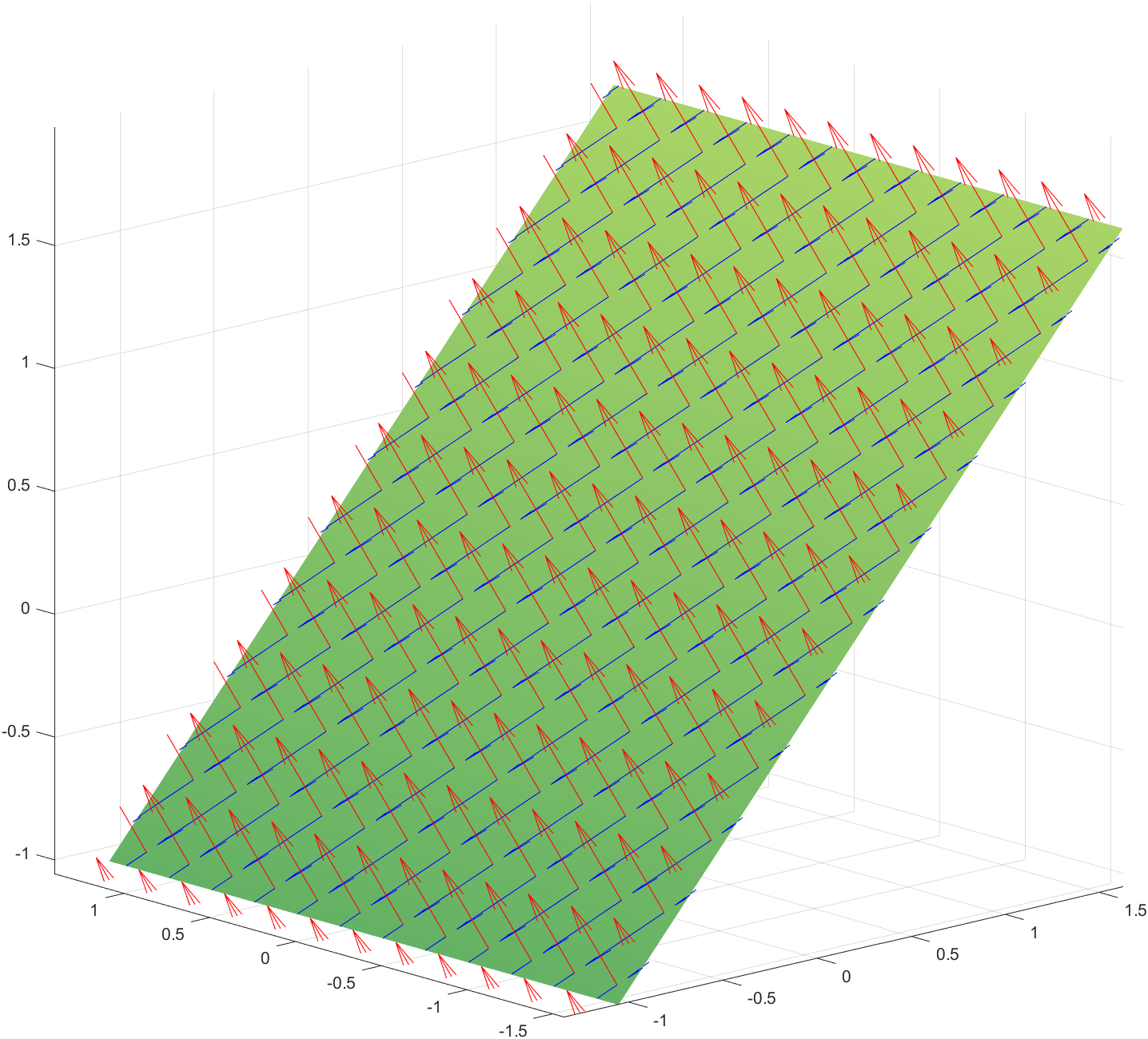}
\centering{\caption{A planar constant angle null hypersurface.}}
\end{figure}

\begin{example}\label{example:LightCone}
We now provide some examples of constant angle null hypersurfaces in the two-dimensional light cone $\Lambda^2_1$. Let $Z$ be a timelike parallel vector field along $\Lambda^{2}_1 \subset \mathbb{R}^{3}_1$. Without loss of generality, we can assume $Z=e_0$. Let $\Phi: \mathbb{R}^2 \rightarrow \Lambda^{2}_1 $ given by  $\Phi(a_1,a_2) = (\sqrt{a_1^2+a_2^2}, a_1,a_2)$ be a graph parametrization of the light cone, then 
\[\xi = \Phi/\sqrt{a_1^2+a_2^2}-2e_0 = \left( -1, \dfrac{a_1}{\sqrt{a_1^2+a_2^2}}, \dfrac{a_2}{\sqrt{a_1^2+a_2^2}} \right)\]
satisfies $g(Z, \xi)=1$. Moreover, by taking the screen distribution corresponding to 
\[N = \left( N_0, \dfrac{a_1(1-N_0) + a_2\sqrt{-1+2N_0}}{\sqrt{a_1^2+a_2^2}},\dfrac{a_2(1-N_0) - a_1\sqrt{-1+2N_0}}{\sqrt{a_1^2+a_2^2}} \right)\]
we obtain a null hypersurface $(\Lambda_2^1,g,S(T\Lambda_2^1))$ having a constant angle with respect to  $Z$.
If $N_0 = \frac{1}{2}$ then 
\[N = \left( \frac{1}{2}, \frac{a_1}{2\sqrt{a_1^2+a_2^2}}, \frac{a_2}{2\sqrt{a_1^2+a_2^2}} \right) =\frac{1}{2} \xi - e_0,\]
which corresponds to the trivial case depicted in Example \ref{example:trivial}. Notice also that the resulting screen is planar.
However, if $N_0 < -\dfrac{1}{2}$, then 
\[S(T\Lambda^{2}_1) = span \{ e_0 + \frac{a_1 \sqrt{-1-2N_0} -a_2}{\sqrt{-1-2N_0}\sqrt{a_1^2+a_2^2}} e_1 + \frac{a_1 + a_2 \sqrt{-1-2N_0}}{\sqrt{-1-2N_0}\sqrt{a_1^2+a_2^2}} e_2 \} ,\]
which is a non trivial and non planar screen (see Figure \ref{fig:cono}).
\end{example}

\begin{figure}[ht!]\label{fig:cono}
\centering \includegraphics[scale=0.25]{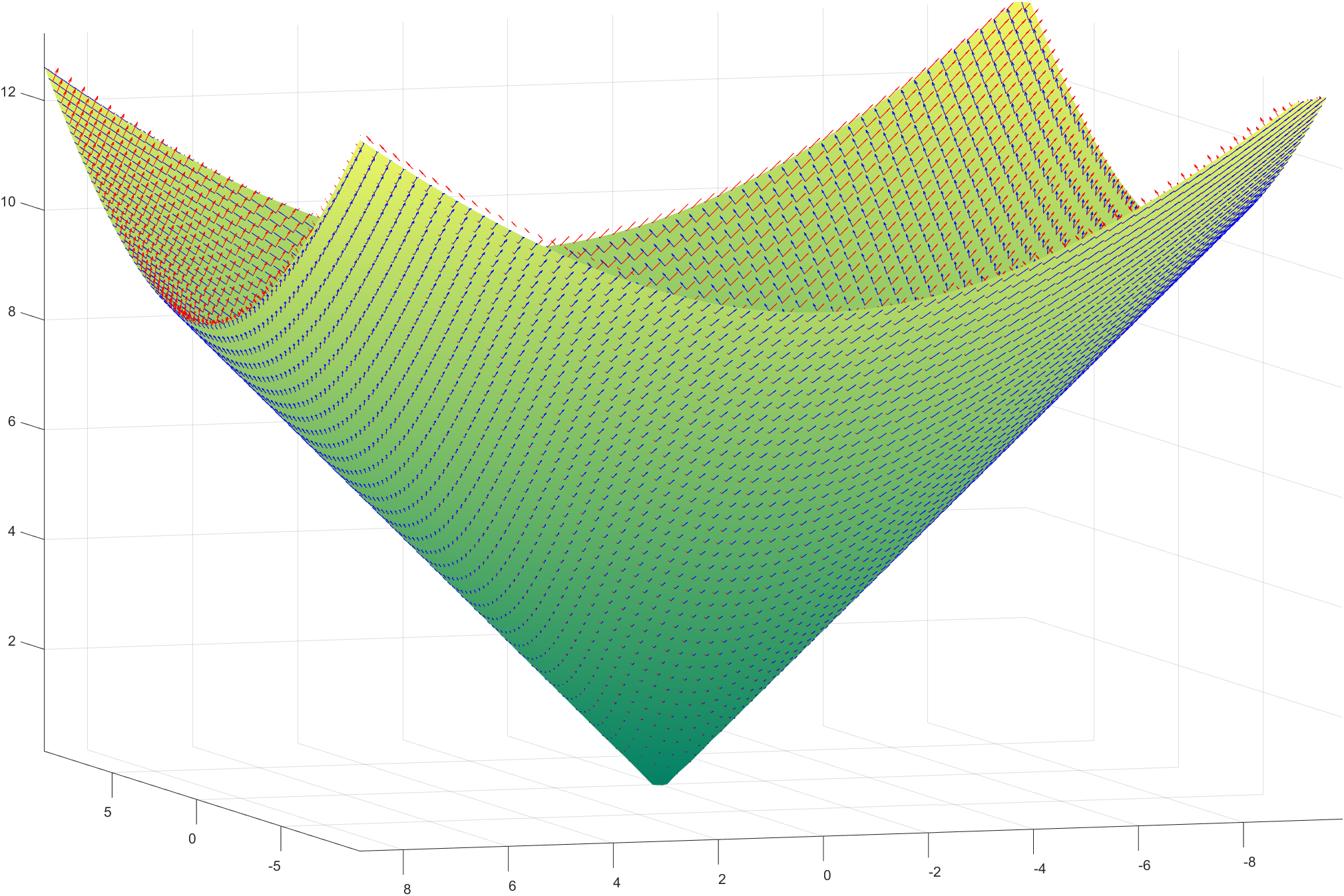}
\centering{\caption{A non planar constant angle null hypersurface.}}
\end{figure}

\section{Canonical principal directions}\label{S2}

Originally proposed by F. Dillen \cite{Dillen10}, the concept of canonical principal direction is closely related to constant angle hypersurfaces. Indeed, in several different settings we have that the preferred vector field gives rise to a principal direction when projected to the hypersurface. 
The precise definition is as follows.

\begin{definition}
Let $(\bar{M},\bar{g})$ be a semi-Riemannian manifold, $(N,h)$ a semi-Riemannian manifold of it, and $Z\in\Gamma (T\bar{M})$ a vector field along $N$. We say $Z$ is a {\em canonical principal direction} if its projection $Z^\top\in \Gamma (TN)$ is a principal direction  for some normal vector field $\mathbf{n}\in \Gamma (TN^\perp)$.   
Let $(\bar{M},\bar{g})$ be a semi-Riemannian manifold, $(N,h)$ a semi-Riemannian manifold of it, and $Z\in\Gamma (T\bar{M})$ a vector field along $N$. We say $Z$ is a {\em canonical principal direction} if its projection $Z^\top\in \Gamma (TN)$ is a principal direction  for some normal vector field $\mathbf{n}\in \Gamma (TN^\perp)$.   
\end{definition}

As it turns out, there is a close connection between CC vector fields and canonical principal directions.  For CC vector fields we have the following equivalences regarding a canonical principal direction.




\begin{lemma}\label{lemma:cpd}
Let $ ( M ,g, S ( T M )) $ be a screen integrable null hypersurface of $ (\overline { M },\overline{g})$  and $Z\in\Gamma (T\bar{M})$ a CC vector field along $M$. The following are equivalent
\begin{enumerate}
	\item $ (M,g,S ( T M )) $ has a canonical principal direction with respect to $A_{Z^\perp} $.
 \item  $\bar{g} (Z,\xi )\bar{g} (N,\xi )$ is constant in directions tangent to $ S ( T M ) $ and orthogonal to $ \overset { * } { Z } $.
 \item There exists a null frame $\{\xi ,N\}$ such that $\bar{g} (Z,\xi )$ and $\bar{g} (N,\xi )$ are constant in directions tangent to $ S ( T M ) $ and orthogonal to $ \overset { * } { Z } $.
	\item $ \bar{ g } ( \overset{*}{Z} , \overset{*}{Z} ) $ is constant in directions tangent to $ S ( T M ) $ and orthogonal to $ \overset { * } { Z } $.
\end{enumerate}
\end{lemma}
\begin{proof}
Notice that the function $\bar{g}(Z,\xi)\bar{g}(Z,\xi )$ is invariant under the scaling of Eq. \ref{eq:gauge}, which readily establishes de equivalence between (2) and (3). Moreover, (2) and (4) are equivalent in virtue of Eq. \eqref{eq:prods} and Lemma \ref{prop:ccc}. Finally (1) $\Leftrightarrow$ (2) by Lemma \ref{lemma:zperp}.
\end{proof}

We now establish the fundamental result that describes the interplay between constant angle null hypersurfaces and canonical principal directions.

\begin{theorem}\label{theoprincdirec}
Let $ ( M , g,S ( T M ) ) $ be a screen integrable null hypersurface of $ (\bar { M } ,\bar{g}) $ and $Z\in\Gamma (T\bar{M})$ a CC vector field. If $ ( M , g,S ( T M ) ) $ has constant angle respect to $ Z $, then $Z$ is a canonical principal direction with respect to $ A_{Z^\perp}$. Moreover, 
\[
	A_{Z^\perp}\overset{*}{Z}= 2\varepsilon_Z \measuredangle (Z,\xi )\measuredangle (Z,N)\varphi \overset{*}{Z}.
\]
\end{theorem}

\begin{proof}
	By Lemmas \ref{ca:easy} and \ref{prop:ccc} we have that $\bar{g}(Z,\xi )\bar{g}(Z,N)$ is constant in directions tangent to $S(TM)$ and orthogonal to $\overset{*}{Z}$, hence $\overset{*}{Z}$ is a principal direction of $A_{Z^\perp}$ in virtue of \ref{lemma:cpd}.
 In order to compute the principal value of $\overset{*}{Z}$, assume $\measuredangle (Z,N)$ is constant. Therefore
		\[ 
		0 = \overset { * } { \nabla } \measuredangle (Z,N) 
		= - \overline { g } ( Z , N ) \frac { 1 } { \lvert Z \rvert ^ 2 } \overset { * } { \nabla } \lvert Z \rvert + \frac { 1 } { \lvert Z \rvert } \overset { * } { \nabla } g ( Z , N ),
		\]
	and thus by Lemma \ref{eqGrads} we have
	\[
	\overset { * } { \nabla } \bar{g} ( Z , N )=\frac{\bar{g}(Z,N)}{\vert Z\vert }\overset{*}{\nabla}\vert Z\vert =\frac{\bar{g}(Z,N)}{\vert Z\vert } \left( \frac { \varepsilon _ Z \varphi } { \lvert Z \rvert } \overset { * } { Z } \right) = \varepsilon_Z\frac{\bar{g}(Z,N)}{\vert Z\vert^2}\varphi\overset{*}{Z},
	\]
and similarly, 
\[
	\overset { * } { \nabla } \bar{g} ( Z , \xi )=\varepsilon_Z\frac{\bar{g}(Z,\xi)}{\vert Z\vert^2}\varphi\overset{*}{Z}.
\]
Thus 
\[
\overset{*}{\nabla}(\bar{g}(Z,\xi )\bar{g}(Z,N))=2\varepsilon_Z\frac{\bar{g}(Z,\xi)\bar{g}(Z,N)}{\vert Z\vert^2}\varphi\overset{*}{Z}=2\varepsilon_Z \measuredangle (Z,\xi )\measuredangle (Z,N)\varphi \overset{*}{Z}
\]
and the result follows from Lemma \ref{lemma:zperp}.
\end{proof}

Recall that for parallel vector field we have $\varphi \equiv 0$
thus we immediately derive from Theorem \ref{theoprincdirec} the following result.

\begin{corollary}\label{cor:AZ0}
Let $ ( M , g,S ( T M ) ) $ be a screen integrable null hypersurface of $ (\bar { M } ,\bar{g}) $ and  $Z\in\Gamma (T\bar{M})$ a parallel vector field along $M$. If $ ( M , g,S ( T M ) ) $ has constant angle respect to $ Z $, then 
\[
A_{Z^\perp}^*\overset{*}{T}=0.
\]
\end{corollary}

\section{Applications}

In this section we analyze two relevant cases. First we focus on parallel vector fields.

\begin{corollary}\label{cor:geodesico}
Let $ ( M , g,S ( T M ) ) $ be a screen inegrable null hypersurface of $ (\bar { M } \bar{M})$, $Z\in\Gamma (T\bar{M})$ a CC vector field along $M$ and $T=\overset{*}{Z}/\vert \overset{*}{Z} \vert$. If $ (M,g,S ( T M ) )$  has a canonical principal direction respect to $A_{Z^\perp}$, then
\begin{enumerate}
\item  $T$  is a geodesic vector field on $S(TM)$.
\item  The eigenvalue of $T$ is 
\[
\lambda =g(\overset{*}{\nabla}_T\overset{*}{Z},T)-\varphi
\]
\end{enumerate}
\end{corollary}

\begin{proof}
Let $\lambda$ be a smooth function such that $A_{Z^\perp}\overset{*}{Z}=\lambda \overset{*}{Z}$. Thus by Lemma \ref{eqcomponents} we have
\begin{eqnarray*}
{\overset{*}{\nabla}}_TT &=&  T\cdot \frac{1}{\vert \overset{*}{Z}\vert  } \overset{*}{Z}+\frac{1}{\vert \overset{*}{Z}\vert }\overset{*}{\nabla}_T\overset{*}{Z}\\
&=& \vert\overset{*}{Z}\vert \, T\cdot \frac{1}{\vert \overset{*}{Z}\vert } T+\frac{1}{\vert \overset{*}{Z}\vert  } \, (\bar{g}(Z,N)A^*_\xi T+\bar{g}(Z,\xi )A_N T+\varphi T)\\
&=& -\vert\overset{*}{Z}\vert \, \frac{1}{\vert \overset{*}{Z}\vert^2}g(\overset{*}{\nabla}_T\overset{*}{Z},T)T +\frac{1}{\vert \overset{*}{Z}\vert  } \, (A_{Z^\perp}T+\varphi T)\\
&=&-\frac{1}{\vert \overset{*}{Z}\vert  } ( g(\overset{*}{\nabla}_T\overset{*}{Z},T)+\lambda +\varphi ) T\\.
\end{eqnarray*}
On the other hand, since $T$ has unit length it follows $\bar{g}({\overset{*}{\nabla}}_TT,T)=0$ which then implies 
\[
-g(\overset{*}{\nabla}_T\overset{*}{Z},T)+\lambda +\varphi =0.
\]
Thus ${\overset{*}{\nabla}}_TT=0$, so $T$ is a geodesic vector field. 
\end{proof}

Notice that Theorem \ref{theoprincdirec} coupled with Corollary \ref{cor:geodesico}  provide insight in the subtle relation between geodesic vector fields and principal directions. Even for the case of surfaces immersed in $\mathbb{R}^3$ the classification of surfaces for which their principal lines are also geodesics is still open. Clearly, in every totally geodesic surface lines of curvature and geodesics agree. However,  in a right circular cylinder every line of curvature is a geodesic, whereas in a regular torus there are curvature lines which are not geodesics, such as parallel curves distinct from equator. 
Examples of surfaces in which lines of curvature are geodesics include Monge and Molding surfaces \cite{BG}. Moreover, in \cite{Ando} N. Ando provides  conditions for a curvature line in a so-called parallel curved surface to be a geodesic. Further he gives a local characterization of the metric in order that a foliation of lines of curvature is given by geodesics.

We now explore two applications to Lorentzian manifolds of dimension $4$.

\begin{theorem}
Let $ ( M ,g, S ( T M ) ) $ be a screen integrable null hypersurface having a constant angle with respect to a parallel vector field $Z\in\Gamma (T\bar{M})$ along $M$. If $\mathrm{tr} A_{Z^\perp}\equiv 0$, then $S(TM)$ is flat.
   \end{theorem}

\begin{proof}
    Since $S(TM)$ is two dimensional, let $W\in\Gamma (S(TM))$ be a unitary vector field orthogonal to $T$. Thus, since $\mathrm{tr} A_{Z^\perp}\equiv 0$  we have that 
    \[
    g(A_{Z^\perp}T,T)+g(A_{Z^\perp}W,W) = 0.
    \]
 In order to prove that $S(TM)$ is flat, we will show that 
        \[
        \overset{*}{\nabla}_T T = \overset{*}{\nabla}_W T = \overset{*}{\nabla}_T W = \overset{*}{\nabla}_W W = 0.
        \]
 By Corollary \ref{cor:geodesico} we have that  $\overset{*}{\nabla}_T T=0$. Further, since $W$ is unitary we have $g(\overset{*}{\nabla}_TW,W)=0$, while $W$ being orthogonal to $T$ implies that
 \[
            0 = T \cdot \overline{g}(W,T)= \overline{ g } (\overset{*}{\nabla}_T T, W) + \overline{ g } (\overset{*}{\nabla}_T W, T) = \overline{ g } (\overset{*}{\nabla}_T W, T), 
            \]
            and consequently $\overset{*}{\nabla}_T W=0$.

Now recall that since $Z$ is parallel we have $\varphi =0$ and $\vert \overset{*}{Z}\vert$ is constant since $(M,g,S(TM))$ has constant angle (see Lemma \ref{ca:easy}).  Hence according to Lemma \ref{eqcomponents} we have
            \[
            \overset { * }{ \nabla } _ W T = \vert \overset{*}{Z}\vert\overset { * }{ \nabla } _ W \overset{Z}{*} =\vert \overset{*}{Z}\vert A_{Z^\perp}W
            \]
 Thus, since $\mathrm{tr} A_{Z^\perp}\equiv 0$ we have
            \[
            g(\overset { * }{ \nabla } _ W T, W) =
            \vert \overset{*}{Z}\vert g(A_{Z^\perp}W,W) =-\vert \overset{*}{Z}\vert g(A_{Z^\perp}T,T)=0,
\]
            where the last equality holds in virtue of Corollary \ref{cor:AZ0}. On the other hand, $T$ being unitary implies that $g  (\overset{*}{\nabla}_W T, T) = 0$ and therefore $\overset { * }{ \nabla } _ W T=0$.

          Finally, since $W$ is unitary we have $g (\overset{*}{\nabla}_W W, W)=0$ while orthogonality guarantees 
\[
0 = W \cdot \overline{g}(W,T)= \overline{ g } (\overset{*}{\nabla}_W W, T) + \overline{ g } (\overset{*}{\nabla}_W T, W) = \overline{ g } (\overset{*}{\nabla}_W W, T),
\]
which proves $\overset { * }{ \nabla } _ W W=0$. Thus, $S(TM)$ is flat as required.
        \end{proof}

For our final application we prove an analogue of Theorem 2.20 in \cite{GG} to the null context. A Lemma is in order.

\begin{lemma}\label{conditions}
Let $ ( M^3 , g, S ( T M ) ) $ be a three dimensional null hypersurface of $ (\bar{ M } ^ 4,\bar{g}) $, $ f $ a smooth function on $M$ and $\{X , Y\} $ an orthonormal frame field of $ S ( T M ) $. If $ \overset { * } { \nabla } _ X X = h X $ for a smooth function $h$, then $ f X $ is a CC on $ S ( T M ) $ if and only if $ f $ satisfies: 
\begin{align*}
Y \cdot f & = 0 \\
X \cdot f + h  f & = f g   ( \overset { * } { \nabla } _ Y X , Y  ) .
\end{align*}
\end{lemma}

\begin{proof}
Direct computations show that
	\begin{align*}
	\overset { * } { \nabla } _ X ( f X ) & = ( X \cdot f ) X + f \overset { * } { \nabla } _ X X  = ( X \cdot f + f \hat{ f } ) X , \\
	\overset { * } { \nabla } _ Y ( f X ) & = ( Y \cdot f ) X + f \overset { * } { \nabla } _ Y X  = ( Y \cdot f ) X + f   g  ( \overset { * } { \nabla } _ Y X , Y ) Y.
	\end{align*}
	The result follows at once. 
\end{proof}

\begin{theorem}\label{theoCMC}
Let $ ( M^3 ,g, S ( T M ) ) $ be a null hypersurface  of a four dimensional Lorentzian space form $ \bar{ M }^4 ( c ) $ of constant sectional curvature $c$ with vanishing null mean curvature.  If the CC vector field $Z\in\Gamma (T\bar{M})$ induces a canonical principal direction with respect to $A_\xi^*$  then $ 1 / \sqrt { \vert 2k^* \vert } T $ is closed and conformal on $ S ( T M ) $, where $ k^* $ is the principal curvature in the direction of $ T=\overset{*}{Z}/\vert \overset{*}{Z}\vert $. Moreover, $k^*$ is constant in directions tangent to $S(TM)$ but orthogonal to $T$.
\end{theorem}

\begin{proof}
Let $\{T,W\}$ be an orthonormal frame of $S(TM)$. Since $A_\xi^*$ is symmetric, by a standard argument $W$ is also a principal direction of $A_\xi^*$. The null mean curvature hypothesis readily implies  $A_\xi^*W=-k^*W$.

Now, since $\{T,W\}$ are orthonormal, by Corollary \ref{cor:geodesico} we have 
$g(T,{\overset{*}{\nabla}}_TW)=-g({\overset{*}{\nabla}}_TT,W)=0$. Since $W$ is unitary, this in turn implies ${\overset{*}{\nabla}}_TW=0$.

Now recall that in a spaceform the curvature endomorphism satisfies
\[
\bar{R}(X,Y)U =\bar{g}(X,U)Y-\bar{g}(Y,U)X
\]
hence by Codazzi equation (refer to Eq. \eqref{eq:Codazzi}) we have

\begin{align*}
0 & = \bar{ g } ( R ( T , W ) T , \xi ) \\
& = T ( B ( T , W ) ) - W ( B ( T , T ) ) - B ( \nabla _ T W , T ) - B( W, \nabla_T T)+ 2B ( \nabla _ W T , T ) \\
& = T ( B ( T , W ) ) - W ( B ( T , T ) ) -  { g } ( A ^ * _ \xi ( T ) , \overset{*} {\nabla}_T W) - \overline { g } ( A ^ * _ \xi ( W ) , \overset{*} {\nabla}_T T)+ 2 { g } ( A ^ * _ \xi ( T ) , \overset{*} {\nabla}_W T)\\
& = T ( B ( T , W ) ) - W ( B ( T , T ) ) + 2 k^* { g } ( T , \overset{*} {\nabla}_W T)\\
& = T \cdot  { g } ( A ^ * _ \xi  T  , W ) - W \cdot  { g } ( A ^ * _ \xi  T  , T ) \\
& = - W \cdot k^*,
\end{align*}
which establishes the second claim of the theorem.

In order to prove the first assertion, we rely on Lemma \ref{conditions} and prove that $h:=0$, $ f : = 1 / \sqrt { \vert k^* \vert } $, $ X : =  T $ and $ Y : = W $ satisfy the hypothesis of the Lemma.  Notice that with these choices, $W\cdot k^*=0$ immediately yields $Y\cdot f=0$ and the first condition of Lemma \ref{conditions} holds. 
	
To verify the second condition of the Lemma, we apply Codazzi equation once again to obtain 
	\begin{align*}
	0 & = \overline { g } ( R ( T , W ) W , \xi ) \\
	& = T ( B ( W , W ) ) - W ( B ( T , W ) ) - B ( \nabla _ T W , W )\\
 &\qquad\qquad - B ( W , \nabla _ T W ) + B ( \nabla _ W  T , W ) + B ( T , \nabla _ W W ) \\
	& = T \cdot g( A ^ * _ \xi  W  , W ) + g( A ^ * _ \xi  W  , \overset { * } { \nabla } _ W T ) + g( A ^ * _ \xi  T  , \overset { * } { \nabla } _ W W ) \\
	& =- T \cdot k^* - g( k^* W , \overset { * } { \nabla } _ W T ) + g( k^* T , \overset { * } { \nabla } _ W W ) \\
	& = -T \cdot k^* - 2k^* g( W , \overset { * } { \nabla } _ W T ).
	\end{align*}
	Thus
\[
-\frac{T \cdot k^*}{2k^*} = g( W , \overset { * } { \nabla } _ W T ),
\]	
 so we have, where $\sigma =\textrm{sign}k^*$,
	\[
	T \cdot f  = - \frac { 1 } { 2 } \frac {  T \cdot ( \vert k^*\vert ) } { ( \vert k^*\vert ) ^ { 3 / 2 } }  = - \frac { 1 } { 2 } \frac { \sigma T \cdot (  k^* ) } { \sigma k^* ( \vert k^*\vert ) ^ { 1 / 2 } }= f   g( Y , \overset { * } { \nabla } _ Y X  )
	\]
	which establishes the second condition of Lemma \ref{conditions} thus completing the proof.
\end{proof}

\section*{Acknowledgments}

 S. Chable-Naal recognizes the support of Conacyt under the Becas Nacionales program (793510). M. Navarro was partially supported by Conacyt SNI  25997. D. Solis was partially supported by Conacyt SNI 38368. M. Navarro y D. Solis were partially supported by grant UADY-FMAT PTA 2023.

\medskip


\begin{enumerate}
\item \textsc{Samuel Chable-Naal. Facultad de Matem\'aticas. Universidad Aut\'onoma de Yuct\'an, Perif\'erico Norte 13615. M\'erida, M\'exico. } samuelchable@alumnos.uady.mx

\item \textsc{Matias Navarro. Facultad de Matem\'aticas. Universidad Aut\'onoma de Yuct\'an, Perif\'erico Norte 13615. M\'erida, M\'exico. } matias.navarro@correo.uady.mx

\item \textsc{Didier A. Solis. Facultad de Matem\'aticas. Universidad Aut\'onoma de Yuct\'an, Perif\'erico Norte 13615. M\'erida, M\'exico. }  didier.solis@correo.uady.mx

\end{enumerate}

\end{document}